\documentclass[12pt]{amsart}

\usepackage{cite}
\usepackage{amsmath}
\usepackage{amstext}
\usepackage{amssymb}

\usepackage{esint}

\usepackage{amsthm}

\usepackage{bm}

\theoremstyle{plain}
\newtheorem{theorem}{Theorem}[section]
\newtheorem{lemma}[theorem]{Lemma}

\theoremstyle{definition}
\newtheorem{definition}[theorem]{Definition}

\numberwithin{equation}{section}

\DeclareMathOperator{\Rp}{Re}
\DeclareMathOperator{\Ip}{Im}

\usepackage{cite}

\begin{document}

\title[Harmonic Conjugation on Weighted Bergman Spaces]%
{Boundedness of Harmonic Conjugation on Weighted Bergman Spaces}
\author{Timothy Ferguson}
\address{Department of Mathematics\\University of Alabama\\Tuscaloosa, AL}
\email{tjferguson1@ua.edu}

\date{\today}

\begin{abstract}
  We prove that if a weight is a Bekoll\'{e}-Bonami weight for some
  $q$ and it satisfies another simple condition that depends on
  $0 < p < \infty$, then the
  operator taking a function to its harmonic conjugate is bounded
  on the harmonic Bergman space $a^p$.  One part of our results
  uses a certain special type of good lambda inequality.  
\end{abstract}

\maketitle

\section{Introduction}
The $p^{\textrm{th}}$ integral mean of a continuous function defined on the unit
disc is
\[
  M_p(r;f) = \left( \int_0^{2\pi} |f(re^{i\theta})|^p \, d\theta \right)^{1/p}.
\]
The Hardy space $H^p$ is the space of all analytic functions in the disc
with bounded integral means; the harmonic Hardy space $h^p$ is the space
of all harmonic functions in the unit disc with bounded integral means.

The Bergman space $A^p$ is the space of all analytic functions in the unit
disc that are in $L^p$ with respect to area measure; the harmonic
Bergman space is defined similarly.  

It is well known that if a function $u$ is in the harmonic Hardy space
$h^p$ for $1 < p < \infty$, then its harmonic conjugate $v$ is as well.
This is equivalent to the boundedness of the Hilbert transform on
$L^p$.  However, the result fails for $0 < p \leq 1$ and for $p = \infty$.  

Surprisingly, if $u$ is in the harmonic Bergman space $a^p$ for
$0 < p < \infty$, then its harmonic conjugate $v$ is as well.  This
was first proved by Hardy and Littlewood \cite{hlc}.  They omitted
some details, which were filled in by Watanabe \cite{wata}.

Recently, the author and Wang \cite{wang} proved that the same result
holds for certain variable exponent Bergman spaces.  In this paper,
we prove that the result holds for weighted Bergman spaces where the
weights are Bekoll\'{e}-Bonami weights with a certain additional property.
The first part of our proof is similar to previous work, and uses an
inequality basically found by Hardy and Littlewood \cite{hlc}.
The second part is novel, and proceeds from a certain special type of
good lambda inequality.  Similar good lambda inequalities should be
relevant to other problems involving harmonic analysis and Bergman spaces. 

A weight is called a Bekoll\'{e}-Bonami weight (for the exponent $p$)
\cite{bb78, bb82} if
there is a constant $0 < C < \infty$ such that for every disc $B$ whose
closure intersects the boundary of the unit disc, one has
\[ \left( \frac{1}{|B|} \int_{B} w(x) \, dx \right) \left( \frac{1}{|B|} \int_{B} w(x)^{-1/(p-1)} \, dx \right)^{p-1} \leq C .\]
One may use other similar definitions.  For example, the discs may be
replaced by Carleson squares, or one may require the inequality holds
for all discs whose double intersects the boundary of the unit disc.

The Bekoll\'{e}-Bonami maximal function is defined by 
\[
  mf(z) = \sup_B \frac{1}{|B|} \int_B |f(w)| \, dA(w)
\]
where the supremum is taken over all balls $B$ containing $z$ whose
closure intersects the boundary of the unit disc.  Other equivalent
definitions are possible.
For $1 < p < \infty$, the Bekoll\'{e}-Bonami maximal function is
bounded from $L^p(w \, dA)$ to $L^p(w \, dA)$ if
$w$ is a Bekoll\'{e}-Bonami weight.  

The method used in this paper is to prove that 
if $f$ is the analytic completion of $u$ and $\|\cdot\|$ is the
norm for certain weighted Bergman spaces, then
\(
\|f'(z) (1-|z|)\| \leq C \|u\|
\)
and
\(
\|f - f(0)\| \leq C \|f'(z) (1-|z|) \|.
\)

In fact, the methods of this paper give certain results that are
independent of the spaces under consideration, and it is possible
to adapt these methods 
to more general spaces than weighted Bergman spaces.  

\section{Control of $f'$ by $u$}

We first focus on the control of $f'$ by $u$.  We will need the following
series of lemmas.  
\begin{lemma}
  Let $u$ be harmonic, in some convex set, and let $f$ be its analytic
  completion.  Then for any points $z$ and $w$ in the set, we have 
  \[ u(w) - u(z) =
    \Rp [f'(z) (w-z)] + \int_0^1 \Rp [f''(z + s(w-z)) (w-z)^2] (1-s) \, ds.\]
\end{lemma}
\begin{proof}
  This is basically an application of Taylor's formula.  
Observe that  $f' = u_x + i v_x = u_x - i u_y$ and $f'' = u_{xx} - i u_{xy}$.
For $a^x$ and $a^y$ real numbers let 
\[
  \partial_{a^x + i a^y} f = \nabla( f) \cdot \langle a^x, a^y \rangle =
  a^x f_x + a^y f_y.
\]
Let $a = a^x + i a^y$.  
Then $\partial_a u = u_x a^x + u_y a^y$ and 
\[
  \partial_a^2 u = u_{xx} (a^x)^2 + 2 u_{xy} a^x a^y + u_{yy} (a^y)^2.
  \]
Notice that $\Rp[a f'] = u_x a^x + u_y a^y$ 
and 
\[ a f'' = u_{xx} a^x + u_{xy} a^y - i (u_{yy} a^y + a^x u_{xy}) \]
using the fact that $u$ is harmonic.
Multiplying the above by $a^x + i a^y$ again and taking the real part gives 
$u_{xx} (a^x)^2 + u_{xy} a^x a^y + a^x a^y u_{xy} + (a^y)^2 u_{yy}.$  
So \[
  \Rp [a^2 f''] = \partial_a^2 u.
  \]

Now given $g(t) = u(z + t(w-z))$, we have by Taylor's formula that 
\[ g(t) = g(0) + g'(0) t + \int_0^t g''(s) (t-s) \, ds.\]
Now notice that $g' = \partial_{w-z} u$ and $g'' = \partial^2_{w-z} u$.
By what we have said, we see that 
\[ g(1) - g(0) = \Rp [f'(z) (w-z)] +
  \int_0^1 \Rp [f''(z+s(w-z)) (w-z)^2] (1-s) \, ds.\]
Therefore 
\[ u(w) - u(z) = \Rp [f'(z) (w-z)] +
  \int_0^1 \Rp [f''(z + s(w-z)) (w-z)^2] (1-s) \, ds.\]  
\end{proof}

\begin{lemma}\label{lemma:ucontrolderivfpointwise}
  Let $\sigma > 0$ and $0 < \eta < 1$.  Assume $f$ is analytic inside
  and on a disc of radius $\sigma$ with center $a$ and
  $|h| = \eta \sigma$ and
  $|z-a| < h$.  Let $M$ be the maximum of $|f'|$ (note the
  derivative) on the boundary of the disc with center $a$ and radius $\sigma$,
  and let $u$ be the real
  part of $f$.  Then
\[|f'(z)| \sigma \leq \frac{1}{\eta} \left( |u(a+h)| + |u(a+ih)| +
    2|u(a)| \right) + \frac{10 \eta}{2(1-\eta^2)} \sigma M. \]
\end{lemma}
Note that this lemma is very similar to but slightly more general than
one by Hardy and Littlewood \cite{hlc}.
Our proof however is different than the
proof of Hardy and Littlewood since they use a power series method.  
\begin{proof}
The previous lemma implies that 
\[ \begin{split} u(a+h) - u(a) &= u(a+h) - u(z) - (u(a) - u(z)) = \\ \Rp f'(z) h + 
\int_0^1 &\Rp f''(z + s(a+h - z)) (a+h - z)^2 (1-s) + \\ 
&\Rp f''(z + s(a - z)) (a - z)^2 (1-s) \, ds\end{split}.\]

Suppose
that on the segment from $a$ to $z$ and on the segment from $z$ to
$a+h$ we have that the second derivative is at most $C$ in absolute
value.  Then the integral is at most $5C |h|^2 / 2$.  In particular,
since $|f'|$ is bounded by $M$ on the circle of radius $\sigma$
centered at $a$, 
it follows from Cauchy's integral formula applied to give $f''$ in terms of
an integral of $f'$ over the boundary of the circle that $f''$ is bounded by
\[ \sigma^{-1} (1 - \eta^2)^{-1} M.\]
So we have that 
\[ |u(a+h) - u(a) - \Rp f'(z) h| \leq
  5 \sigma^{-1} (1 - \eta^2)^{-1} M |h|^2 / 2.
\]
Rearranging gives 
\[
  |\Rp (f'(z) h)| \leq |u(a+h)| + |u(a)| + 5 \sigma^{-1} (1 - \eta^2)^{-1} M |h|^2 / 2\]
or 
\[ |\Rp (f' h/|h|)| \sigma \leq \frac{1}{\eta} \left( |u(a+h)| +
    |u(a)| \right) +  \frac{5\eta}{2(1-\eta^2)} \sigma M. \]
Now take the original inequality with $ih$ in place of $h$ to see that
\[ |\Ip (f' h/|h|)| \sigma \leq \frac{1}{\eta} \left( |u(a+ih)| +
    |u(a)| \right) +  \frac{5\eta}{2(1-\eta^2)} \sigma M. \]
Thus
\[|f'(z)| \sigma \leq
  \frac{1}{\eta} \left( |u(a+h)| + |u(a+ih)| + 2|u(a)| \right)
   +  \frac{10 \eta}{2(1-\eta^2)} \sigma M. \]
\end{proof}

Even though we are studying weighted Bergman spaces, this part of the
paper actually applies under much more general conditions.  We thus
make the following definitions.
\begin{definition}
We say that $\|\cdot\|$ satisfies the triangle inequality up to a constant
if there is a constant $C$ such that
$\|f + g \| \leq C(\|f\| + \|g\|)$.
\end{definition}

\begin{definition}
Suppose that, for each measurable function $f$ in the unit disc, the
quantity $\|f\|$ is nonnegative or infinity, and that it is $0$ only for
the zero function.  
Assumer furthermore that it satisfies the triangle inequality up to
a constant. 
Also suppose that $\|\alpha x\| \leq \|x\|$ if $|\alpha| \leq 1$ and that
$\lim_{\alpha \rightarrow 0} \|\alpha x\| = 0$ for all $x$.
Then $\|\cdot\|$ is called a delta norm.  We call it a uniform delta
norm if
\[ \lim_{\alpha \rightarrow 0} \sup_{x \neq 0} \frac{\|\alpha x\|}{\|x\|}
  =0.
\]
\end{definition}
Note that the modified triangle inequality implies that
$\|2x\| \leq 2C \|x\|$ for all $x$.
If $\|\cdot\|$ is a uniform delta norm, there is a constant
$0<\alpha<1$ such that $\|\alpha x\| \leq \|x\|/2$.
Now let $n$ be the least integer such that $\alpha^n \leq 1/2$.
Then we have $\|x/2\| \leq \|x\|/2^n$.  Letting $x=2y$ gives
$\|y\| \leq \|2y\|/2^n$ or
$\|2y\| \geq 2^n \|y\|$ for all $y \neq 0$.  Thus, there are
constants $a,b > 0$ such that
\[
  \alpha^a \|x\| \leq \|\alpha x \| \leq \alpha^b \|x\|
\]
for $\alpha \geq 1$.  

If $\alpha < 1$ then $\|x\| = \|(1/\alpha) \alpha x\| \leq
\alpha^{-b} \|\alpha x\|$.  Similarly
$\|x\| \geq \alpha^{-a} \|\alpha x\|$.  This implies that
\[
  |\alpha|^b \|x\| \leq \|\alpha x\| \leq \alpha^a \|x\|
\]
for $|\alpha| < 1$.

\begin{definition}
  Suppose that $\|\cdot\|$ is a uniform delta norm, and that $T$ is an
  operator.   
Say that $T$ is bounded on
$\|\cdot\|$ if for each $D > 0$ there is a constant $D'$ such that
$\|Tf\| \leq D$ whenever $\|f\| \leq D'$.
\end{definition}
Now suppose that it is known that
for some $D'$ and some $D$ one has $\|Tf\| \leq D'$ if
$\|f\| \leq D$.  Also suppose that $T$ commutes with scalar
multiplication.
Suppose that $\|f\| < \infty$ and $\|f\| > D$.  Then
$\|(D/\|f\|)^{1/a} f\| \leq D$ and
\[
  \|Tf\| = \| (\|f\|/D)^{1/a} T[(D/\|f\|^{1/a})f] \|
  \leq (\|f\|/D)^{b/a} D'.
\]
Thus $T$ is bounded.  

Define the function $\delta_b(z) = 1 - |z|$.  Thus
$\delta_b(z)$ is the distance from $z$ to the boundary of the unit disc.  
 We now show that $\|u\|$ controls $\|\delta_b f'\|$
in certain cases.
Define
\[
  M_c(g)(z) = \sup_{\{w: |w-z| < \delta_b(z)/2\}} |g(w)|.
\]
\begin{theorem}\label{thm:ucontrolderiva}
  Suppose that $\|\cdot\|$ is a uniform delta norm.  Also suppose that
  \begin{enumerate}
  \item%
    The operator $M_c$ is bounded on $\|\cdot\|$ when restricted to
    functions of the form $\delta_b g$ where $g$ is analytic.  
  \item
    for sufficiently small $0 < \eta < 1$ 
    the operators
    \[u \mapsto u(z + e^{i \arg(z)} \eta \delta_b)
      \text{ and }
    u \mapsto u(z + i e^{i \arg(z)} \eta \delta_b)  
    \] are bounded. 
  \end{enumerate}
  Then the operator taking $u$ to $\delta_b f'$ is bounded.  
\end{theorem}
\begin{proof}
  First notice that for any function $g$, one has
  $M_c( \delta_b g)(z) \approx \delta_b(z) (M_c g)(z)$ since
the function $\delta_b(w)$ is essentially constant inside the disc
$|w-z| < \delta_b(z)/2$.  Thus
\[
\| \delta_b M_c g\| \leq C \|M_c(\delta_b g)\|.
\]
By assumption (1), there is a constant $C'$ such that if
$\|\delta_b g\| \leq 1$ then $\|M_c(\delta_b g)\| \leq C'$.  Thus in this
case $\| \delta_b M_c g\| \leq C C'$, so since $M_c$ commutes with
scalar multiplication, the operator taking $\delta_b g$ to
$\delta_b M_c g$ is bounded.  
Thus, the operator taking $f' \delta_b$ to $M_c(f') \delta_b$ is bounded.

In Lemma \ref{lemma:ucontrolderivfpointwise},
take $\sigma = \delta_b/2$.  
  The fact that $\|\cdot\|$ satisfies
  the triangle inequality up to constant implies that, for some
  $C>0$, we have 
  \[
    \begin{split}
  \|f' \delta_b/2\|
    \leq
    & \frac{C}{\eta} \left( \|u(a+h(z))\| + \|u(a+ih(z))\| + 2\|u(a)\| \right) 
    \\
    & \qquad
    + C\left\|\frac{10\eta}{2(1-\eta^2)} (\delta_b/2) M_c(f')\right\|.
  \end{split}
\]
Here $h(z) = \eta (\delta_b/2) e^{i \arg z}$.  
Thus
  \[
    \begin{split}
  \|f' \delta_b\|
    \leq
    & \frac{C}{\eta} \left( \|u(a+h)\| + \|u(a+ih)\| + 2\|u(a)\| \right) 
    \\
    & \qquad
    + C\left\|\frac{10\eta}{2(1-\eta^2)} \delta_b M_c(f')\right\|
  \end{split}
\]
where the constant $C$ may be different than above.

By assumption, there is a constant $N'$ 
depending only on $N = \|f' \delta_b\|$
such that $\|\delta_b M_c(f')\| \leq N'$. Since
$\|\cdot\|$ is a uniform delta norm, if we choose $\eta$ small enough,
we can ensure that
\[
  C \left\| \frac{10 \eta}{2 (1-\eta^2)}
      \sigma M_c(f') \right| \leq \frac{N}{2}.
\]
But this ensures that
\[
  \frac{N}{2} \leq \frac{C}{\eta} \left( 2\|u\| + \|u(z + h(z))\| +
    \|u(z + i h(z)\| \right).
\]
By assumption, this implies that
\[
  \frac{N}{2} \leq C(\eta) \|u\|
\]
where $C(\eta)$ may depend on $\eta$.  This holds for all $\eta$ small
enough.  So we may choose any $\eta$ that is small enough, and the
theorem is proven.  
\end{proof}

Consider $\delta_b M_c(f)$, which is equivalent up to a constant
multiple to $M_c (\delta_b f)$.  We now wish to show that
\[
  \| \delta_b M_c(f') \|_p \leq \|\delta_b f'\|_p, 
\]
where $\| \cdot \|_p$ is the $L^p(w \, dA)$ norm for some
weight $w$.  
But this is equivalent to showing that
\[
  \|[M_c(\delta_b f')]^{p/q}\|_q \leq
  C \| \delta_b^{p/q} |f'|^{p/q} \|_q, 
\]
or that
\[
  \|M_c(\delta_b^{p/q} |f'|^{p/q})\|_q \leq
  C \| \delta_b^{p/q} |f'|^{p/q} \|_q
\]
But by the subharmonicity of $|f'|^{p/q}$ and the fact that
$\delta_b$ is essentially constant on the discs of radius
$\delta_b(z)/2$, one sees that
$M_c(\delta_b^{p/q} |f'|^{p/q})$ is bounded by a constant times
$m(\delta_b^{p/q} |f'|^{p/q})$, where $m$ is the Bekoll\'{e}-Bonami
maximal function.  As long as this operator is bounded on
$L^q(w\, dA)$, one can
conclude that
\[
  \| \delta_b M_c(f') \|_p \leq C \|\delta_b f'\|_p.
\]

Thus, we have the following theorem
\begin{theorem}\label{thm:ucontrolderiv}
  Suppose that $w$ is a Bekoll\'{e}-Bonami weight for some $q$ and let
  $\|\cdot\| = \|\cdot\|_{L^p(w \, dA)}$.
  Suppose that 
    for sufficiently small $0 < \eta < 1$ 
    the operators
    \[u \mapsto u(z + e^{i \arg(z)} \eta \delta_b)
      \text{ and }
    u \mapsto u(z + i e^{i \arg(z)} \eta \delta_b)  
  \]
   are bounded. 
  Then the operator taking $u$ to $\delta_b f'$ is bounded.  
\end{theorem}

\section{Control of $f$ by $f'$}

A dyadic Carleson square is a set where $|z|$ is between
$1-2^{-n}$ and $1$ for some $n \geq 0$, and the argument
lies in some interval of length $2\pi 2^{-n}$.
The top half is the set where $|z|$ is in the left half of the given
interval; that is, where $1-2^{-n} \leq |z| < 1-2^{-n-1}$.  

Divide the unit disc into nested dyadic Carleson squares, the first one of
which is the entire disc.
Then the top halves of these squares will be disjoint.

If $C$ is a disjoint union of dyadic Carleson squares,
denote by $\mathcal{T}(C)$ the
union of the top halves of the Carleson squares.
If $U$ is any subset of the unit disc, define the shaded set with respect to
$\mathbb{D}$ to be the set
$\mathcal{S}(U) = \{z \in \mathbb{D} : (\exists w \in U)
(\arg(z) = \arg(w) \text{ and } |z| \geq |w|\}$.
So if a point light source is placed at the center of the unit disc, the
points in $\mathcal{S}(U)$ are either in $U$ or are blocked by points in $U$
from seeing the light source.  
It follows that if $C$ is a disjoint union of dyadic Carleson squares,
that $\mathcal{S}(\mathcal{T}(C)) = C$.  
Let $\mathcal{P}(C)$ denote the parent of a given dyadic Carelson square $C$,
and let $\mathcal{P}'(C)$ denote the set of all its ancestors.  

Let $b$ and $d$ be nonnegative functions on the unit disc that are both
constant on the top halves of Carleson squares.  If $T$ is the top
half of a Carleson square, we let $b(T)$ denote the value of $b$ on
$T$, and similarly for $d$.  

If $T$ is the top half of the Carleson square $C$, let 
\[
  B(T) = \max_{P \in \mathcal{P}'(S) \cup \{T\}} b(\mathcal{T}(P)).
\]
Define $D$ similarly.  

Assume that whenever $C$ is a Carleson square 
and $T$ is its top half, we have 
\begin{equation}\label{eq:bdinequality}
  b(T) \leq b(\mathcal{T}(\mathcal{P}(C))) + d(T).
\end{equation}
If $C$ does not have a parent, we interpret the term
$b(\mathcal{T}(\mathcal{P}(C)))$ to equal $0$.  
We take
$b(T)$ to be the supremum of an analytic function $|f|$ on $T$, and
$d(T)$ to be $2^{-k-1}$ times the supremum of $|f'|$ on $T$, where 
$1 - 2^{-k} \leq |z| < 1-2^{-k-1}$ on $T$.  We will also assume that
$f(0) = 0$.  Then
inequality \eqref{eq:bdinequality} follows from the
fundamental theorem of calculus.  However, all of the relations
we give below between $b$, $B$, $d$ and $D$ do not depend on 
the definitions in terms of $f$ given above. 

Let $T_k$ be a sequence of top halves of Carleson squares, where
each square is the parent of the next one, and define $b_k = b(T_k)$,
and similarly for $d$.  
From the fact that 
\[
  b_{k+1} \leq b_k + d_{k+1}
\]
we may deduce that
\[
  b_{k+n} \leq  b_k + n (\max_{k < j \leq k+n} d_j)
  \leq B_k + n D_{k+n}.
\]
Since for $0 \leq n' \leq n$ we have
\[
  b_{k+n'} \leq B_k + n' D_{k+n'} \leq B_k + n D_{k+n}
\]
it follows that 
\begin{equation} \label{eq1} B_{k+n} \leq B_k + n D_{k+n}. \end{equation}

We now show this gives a good lambda inequality between $B$ and $D$.    
Let $\lambda > 0$ be given and suppose that $B_m \geq \lambda$ and
$B_j < \lambda$ for $j < m$.  Also suppose that
$B_n \geq \alpha \lambda$ but $D_n \leq \gamma \lambda$.
Now apply \eqref{eq1} with $m-1$ in place of $k$ and $n-(m-1)$ in place
of $k$ to get 
\[
  \alpha \lambda \leq \lambda + \gamma \lambda (n - m + 1)
\] so that
\[(\alpha - 1)/\gamma \leq (n-m + 1)\]
or
\begin{equation}
  \label{eq2} n - m \geq \frac{\alpha - 1}{\gamma} - 1 
\end{equation}

For a given Carleson square $C$, let $\mathcal{C}^n(C)$ denote the
set of all descendants of the square $n$ levels below it.
So returning to the functions $B$ and $D$ above, let $C$ be any Carleson
square where $B \geq \lambda$ on its top half but
$B < \lambda$ on its parent's top half.  We have 
\begin{equation} \label{eq3}
  \{z \in C: B \geq \alpha \lambda; D \leq \gamma \lambda\} \subseteq \mathcal{C}^{\lceil (\alpha - 1)/\gamma \rceil-1}(C).
\end{equation}
This gives a good lambda inequality with respect to Lebesgue measure: 
\[
  m \{z \in C: B \geq \alpha \lambda; D \leq \gamma \lambda\}
  \leq
  2^{1-\lceil (\alpha - 1)/\gamma \rceil} m(C).
\]
Now suppose that the top half of any square has $\mu$ measure at least
$\tau$ times the $\mu$ measure of the square.
In particular, this happens if $\mu = w \, dA$ where $w$
is a Bekoll\'{e}-Bonami weight for some $q$.  
Then the bottom half has $\mu$ measure
at most $1 - \tau$ times the $\mu$ measure of the square, so
\[
  \mu \{z \in C: B \geq \alpha \lambda; D \leq \gamma \lambda\}
  \leq
  (1-\tau)^{\lceil (\alpha - 1)/\gamma \rceil-1} \mu(C).
\]  
This is still a good lambda inequality, and the constant can be made
arbitrarily small by making $\gamma \rightarrow 0$.
By a well known argument \cite{garnettbook, steinbigbook},
this implies that there is a constant
$K$ for which
$\|B\| \leq K  \|D\|$.

Now consider a Carleson square $C$ where $D \geq \lambda$ on its top half;
suppose $D < \lambda$ on the top half of its parent.
Then $d \geq \lambda$ on the top half of $C$.  
Now suppose again that $\mu$ is a
measure for which the measure of any top half
of a Carelson square is at least $\tau$ times the measure of the square,
where $\tau > 0$.
Then $\tau \mu(C) \leq \mu \{z \in C: d(z) \geq \lambda\}$.
This implies that
\[
  \tau \mu\{z: D(z) \geq \lambda \} \leq \mu\{z: d(z) \geq \lambda\}.
\]
Therefore, by the layer cake decomposition, one sees that
\[
  \tau \int D(z)^p \, d\mu \leq \int d(z)^p \, d\mu.
\]

From what we have said before, this implies that
\[
  \|B\| \leq K' \|d\|
\]
for some constant $K$.

It is clear that $|f| \leq B$ pointwise.  Also, from what we said
in the previous section, $\|d\|_p$ is bounded by a constant times
$\|f' \delta_b\|_p$.  Thus we have the following theorem. 

\begin{theorem}\label{thm:derivcontrolf}
  Suppose that $0 < p < \infty$ and that $w$ is a
  Bekoll\'{e}-Bonami weight for some $q$.  Then
  the operator taking $f' \delta_b$ to $f-f(0)$ is bounded on
  $\|\cdot\|_{L^p(w \, dA)}$.
\end{theorem}

Putting together Theorems \ref{thm:ucontrolderiv} and
\ref{thm:derivcontrolf} gives the following result.  
\begin{theorem}\label{thm:ucontrolv}
  Let $0 < p < \infty$ and 
    suppose that $w$ is a Bekoll\'{e}-Bonami weight for some $q$ and let
  $\|\cdot\| = \|\cdot\|_{L^p(w \, dA)}$.
  Suppose that 
    for all sufficiently small $0 < \eta < 1$ 
    the operators
    \[u \mapsto u(z + e^{i \arg(z)} \eta \delta_b)
      \text{ and }
    u \mapsto u(z + i e^{i \arg(z)} \eta \delta_b)  
  \]
  are bounded. 
  Then the operator taking $u$ to its harmonic conjugate $v$ is bounded.  
\end{theorem}

Note that by a simple change of variables in the integrals defining
$\|u(z + e^{i \arg(z)} \eta \delta_b(z))\|$ and
$\|u(z + i e^{i \arg(z)} \eta \delta_b)\|$, one sees the conditions of the
theorem apply if the Bekolle\'{e}-Bonami
weight is essentially constant on the top half of
each Carleson square.

\providecommand{\bysame}{\leavevmode\hbox to3em{\hrulefill}\thinspace}
\providecommand{\MR}{\relax\ifhmode\unskip\space\fi MR }
\providecommand{\MRhref}[2]{%
  \href{http://www.ams.org/mathscinet-getitem?mr=#1}{#2}
}
\providecommand{\href}[2]{#2}

\end{document}